\documentclass[12pt]{amsart}

\usepackage{amsmath,amssymb,amscd,amsthm,graphicx,enumerate}
\usepackage[usenames,dvipsnames]{xcolor}

\usepackage{hyperref}
\hypersetup{breaklinks=true}

\numberwithin{equation}{section}
\theoremstyle{plain}



\newtheorem{theorem}[equation]{Theorem}

\newtheorem{proposition}[equation]{Proposition}
\newtheorem{lemma}[equation]{Lemma}
\newtheorem{corollary}[equation]{Corollary}

\theoremstyle{remark}

\theoremstyle{definition}
\newtheorem{definition}[equation]{Definition}

\newtheorem*{question*}{Question}

\newcommand{\M}{{\mathcal M}}

\newcommand{\R}{\mathbb R}

\renewcommand{\t}{\mathfrak{t}}
\newcommand{\T}{{\mathcal T}}

\newcommand{\V}{{\mathcal V}}

\newcommand{\Z}{\mathbb Z}

\newcommand{\ben}{\begin{enumerate}}
\newcommand{\een}{\end{enumerate}}

\newcommand{\const}{\operatorname{const.}}
\newcommand{\cyl}{Cyl}

\newcommand{\diam}{\operatorname{Diam}}

\newcommand{\dvol}{\operatorname{dvol}}

\definecolor{gray}{gray}{0.7}

\newcommand{\ric}{\operatorname{Ric}}

\newcommand{\vol}{\operatorname{vol}}

\newcommand{\D}{\partial}

\newcommand{\eps}{\epsilon}
\newcommand{\ga}{\gamma}

\newcommand{\la}{\lambda}

\newcommand{\ra}{\rightarrow}

\newcommand{\Rm}{\operatorname{Rm}}

\def\XXint#1#2#3{{\setbox0=\hbox{$#1{#2#3}{\int}$}
     \vcenter{\hbox{$#2#3$}}\kern-.5\wd0}}

\begin{document}

\begin{abstract}
We establish several quantitative results about singular Ricci flows, including estimates on the curvature and volume, and the set of singular times.
\end{abstract}

\title{Singular Ricci flows II}
\author{Bruce Kleiner}
\address{Courant Institute of Mathematical Sciences\\
251 Mercer St. \\
New York, NY  10012}
\email{bkleiner@cims.nyu.edu}

\author{John Lott}
\address{Department of Mathematics\\
University of California at Berkeley \\
Berkeley, CA  94720}
\email{lott@berkeley.edu}
\thanks{The first author was supported by NSF grants DMS-1405899 and
  DMS-1711556, and a Simons Collaboration grant. The second author was supported by NSF grants DMS-1510192 and DMS-1810700.}
\date{November 17, 2018}
\maketitle

\tableofcontents

\section{Introduction}

In \cite{kleiner_lott_singular_ricci_flows,bamler_kleiner_uniqueness}, it was shown that there exists a canonical Ricci flow through singularities starting from an arbitrary compact Riemannian $3$-manifold, and that this flow may obtained as a limit of a sequence of Ricci flows with surgery.  These results confirmed a conjecture of Perelman \cite{perelman_entropy,perelman_surgery}, and  were used in the proof of the Generalized Smale Conjecture in \cite{bamler_kleiner_gsc}.  

The purpose of this paper, which is a sequel to \cite{kleiner_lott_singular_ricci_flows}, is to further study Ricci flow through singularities.      

We recall that the basic object introduced in \cite{kleiner_lott_singular_ricci_flows} is a {\em singular Ricci flow}, which is a Ricci flow spacetime subject to several additional conditions; see
Definition \ref{srf} of 
Section~\ref{sec2} or \cite[Def. 1.6]{kleiner_lott_singular_ricci_flows}.

In the following, we let $\M$ be a singular Ricci flow
with parameter functions $\kappa$ and  $r$, and we let $\M_t$ denote
a time slice. 
The main results of the paper are the following.

\begin{theorem} \label{t1}
  For all $p \in (0,1)$ and all $t$, the scalar curvature is
  $L^p$ on $\M_t$.
\end{theorem}

\begin{theorem} \label{t2}
  The volume function $\V(t) = \vol(\M_t)$ has a locally bounded
upper-right derivative and is locally $\alpha$-H\"older
  in $t$ for some exponent $\alpha \in (0,1)$.
    \end{theorem}

The first assertion of the theorem was shown in  \cite[Proposition 5.5 and Corollary 7.7]{kleiner_lott_singular_ricci_flows}, so the issue here is to prove
H\"older continuity in the opposite direction.

We state the next result loosely, with a more
precise formulation given later.

\begin{theorem} \label{t3}
  The {\it a priori} assumptions in Definition \ref{srf} of a
  singular Ricci flow are
  really conditions on the spacetime near infinity, in the sense
  that if the conditions hold outside of compact subsets then,
  after redefinition of $\kappa$ and $r$, they hold everywhere.
  \end{theorem}

Finally, we estimate the size of the set of
singular times in a singular Ricci flow.

\begin{theorem} \label{t4}
  For any $T < \infty$, the set of times $t \in [0, T]$ for which
  $\M_t$ is noncompact has Minkowski dimension at most $\frac12$.
  \end{theorem}

 The structure of the paper is as follows.
In Section \ref{sec2} we recall some notation and
terminology from \cite{kleiner_lott_singular_ricci_flows}. In Section \ref{sec3} we prove
some needed results about compact $\kappa$-solutions.
Section \ref{sec_curvature_volume_estimates}
has the proofs of Theorems \ref{t1} and \ref{t2}.
In Section \ref{asympcon} we prove Theorem \ref{t3} and in
Section \ref{sec6} we prove Theorem \ref{t4}.

\section{Notation and terminology} \label{sec2}

We will assume some familiarity with \cite{kleiner_lott_singular_ricci_flows}, but in order to make this paper as self-contained as possible, we will give precise references for all results from \cite{kleiner_lott_singular_ricci_flows} that are used here.
We follow the notation and terminology of \cite{kleiner_lott_singular_ricci_flows}.
All manifolds that arise will be taken to be orientable.
A {\em $\kappa$-solution} is a special type of ancient Ricci flow
solution, for which we refer to \cite[Appendix A.5]{kleiner_lott_singular_ricci_flows}.
The function $r : [0, \infty) \rightarrow (0, \infty)$ is the parameter
$r$ of the canonical neighborhood assumption \cite[Appendix A.8]{kleiner_lott_singular_ricci_flows}.

\begin{definition} \label{def_ricci_flow_spacetime}
  A {\em Ricci flow spacetime} is a tuple $(\M,\t,\D_{\t},g)$ where:
  \begin{itemize}
  \item $\M$ is a smooth manifold-with-boundary.
  \item
    $\t$ is the {\em time function} -- a submersion
    $\t:\M\ra I$ where $I\subset\R$ is  a time interval;
    we will usually take $I=[0,\infty)$.
    \item
      The boundary of
      $\M$, if it is nonempty,
      corresponds to the endpoint(s) of the
      time interval: $\D\M=\t^{-1}(\D I)$.
    \item $\D_{\t}$ is the {\em time vector field}, which satisfies $\D_{\t}\t\equiv 1$.
    \item  $g$ is a smooth inner product on the spatial subbundle
      $\ker(d\t)\subset T\M$, and $g$
      defines a Ricci flow: ${\mathcal L}_{\D_{\t}} g = - 2 \ric(g)$.
  \end{itemize}
\end{definition}

For $0\leq a< b$, we write $\M_a=\t^{-1}(a)$, $\M_{[a,b]}=\t^{-1}([a,b])$ and
$\M_{\leq a}=\t^{-1}([0,a])$.
Henceforth, unless otherwise specified, when we refer to geometric quantities such as
curvature, we will implicitly be referring to the metric on the time slices.

\begin{definition}
\label{srf}
A Ricci flow spacetime
$(\M,\t,\D_{\t},g)$
is a {\em singular Ricci flow} if it is $4$-dimensional,
the initial time
slice $\M_0$ is a compact normalized Riemannian manifold
and
\begin{enumerate}
  \renewcommand{\theenumi}{\alph{enumi}}
\item The scalar curvature function $R:\M_{\leq T}\ra \R$ is
  bounded below and proper
  for all $T\geq 0$.
\item $\M$ satisfies the Hamilton-Ivey pinching condition
  of \cite[(A.14)]{kleiner_lott_singular_ricci_flows}.
\item
  For a global parameter $\epsilon > 0$ and
  decreasing
  functions $\kappa, r:[0,\infty)\ra (0,\infty)$, the spacetime
    $\M$ is
    $\kappa$-noncollapsed below scale $\epsilon$ in the sense of
    \cite[Appendix A.4]{kleiner_lott_singular_ricci_flows}
    and satisfies the $r$-canonical neighborhood assumption
    in the sense of \cite[Appendix A.8]{kleiner_lott_singular_ricci_flows}.
\end{enumerate}
\end{definition} 

Here ``normalized'' means that
at each point $m$ in the initial time slice,
the eigenvalues of the curvature operator $\Rm(m)$ are bounded by one in
absolute value, and the volume of the unit ball $B(m,1)$ is at least
half the volume of the Euclidean unit ball.
By rescaling, any
compact Riemannian manifold can be normalized.
 ``Proper'' has the usual meaning, that the preimage of a compact set is compact.  Since $R$ is bounded below, its properness means that as one goes out an end of $\M_{\leq T}$, the function $R$ goes to infinity.

Let $(\M, \t, \partial_{\t}, g)$ be a Ricci flow spacetime
(Definition \ref{def_ricci_flow_spacetime}).
For brevity, we will often write $\M$ for the quadruple.

Given $s > 0$, the rescaled
Ricci flow spacetime is
$\hat \M(s) = (\M, \frac{1}{s} \t, s \partial_{\t}, \frac{1}{s} g)$.

\begin{definition}
  \label{def_worldline}
  Let $\M$ be a Ricci flow spacetime.
  A path $\ga:I\ra\M$ is {\em time-preserving} if $\t(\ga(t))=t$ for  all
  $t\in I$.
  The {\em worldline} of a point
  $m\in\M$ is the maximal time-preserving integral curve $\ga:I\ra \M$ of the
  time vector field $\D_{\t}$, which  passes through $m$.
\end{definition}

If $\ga:I\ra \M$ is a worldline then we may have $\sup I <\infty$.   In this case,
the scalar curvature blows up along $\ga(t)$ as $t\ra \sup I$, and the worldline
encounters a singularity.  An example would be a shrinking round space form, or a
neckpinch.
A worldline may also encounter a singularity going backward
in time.

\begin{definition}
  A worldline $\ga:I\ra\M$ is {\em bad} if $\inf I>0$, i.e.  if it is not
  defined at $t=0$.
  \end{definition}

Given $m \in M_t$, we write $B(m, r)$ for the open metric
ball of radius $r$ in $\M_t$. We write $P(m,r,\Delta t)$ for the
parabolic neighborhood, i.e. the set of points
$m^\prime$ in $\M_{[t, t + \Delta t]}$ if $\Delta t > 0$
(or $\M_{[t + \Delta t, t]}$ if $\Delta t < 0$) that lie on the
worldline of some point in $B(m,r)$.
We say that $P(m,r, \Delta t)$ is {\em unscathed} if
$B(m,r)$ has compact closure in $\M_t$ and
for every
$m^\prime \in P(m,r, \Delta t)$, the maximal worldline $\gamma$ through
$m^\prime$ is defined on a time interval containing $[t, t+\Delta t]$
(or $[t+\Delta t, t]$).
We write $P_+(m,r)$ for the
forward parabolic ball $P(m,r,r^2)$ and
$P_-(m,r)$ for the
backward parabolic ball $P(m,r,-r^2)$. 

\section{Compact $\kappa$-solutions} \label{sec3}

In this section we prove some structural results about
compact $\kappa$-solutions.
The main result of this section, Corollary
\ref{corcloseness}, will be used in the proof of
Proposition \ref{prop_r_in_lp}.

We recall from \cite[Appendix A.5]{kleiner_lott_singular_ricci_flows}
that if $\M$ is a $\kappa$-solution
then $\M_{t, \hat\eps}$ denotes the points in $\M_t$ that are not
centers of $\hat\eps$-necks.

\begin{lemma} \label{lemcloseness}
  There is some $\overline{\epsilon} > 0$ so that
  for any
  $0 < \hat\eps < \overline{\epsilon}$, there are
  $\epsilon^\prime = \epsilon^\prime(\hat\eps) > 0$ and
  $\alpha = \alpha(\hat\eps) < \infty$
  with the following property.
  Let $\M$ be a compact $\kappa$-solution.  Suppose that
  $\M_t$ contains an $\epsilon^\prime$-neck.  Then there are points
  $m_1, m_2 \in \M_t$ so that $\M_{t,\hat\eps}$ is
  covered by disjoint balls
  $B(m_1, \alpha R(m_1, t)^{- \frac12})$ and
  $B(m_2, \alpha R(m_2, t)^{- \frac12})$, whose intersections with
  $\M_t - \M_{t, \hat\eps}$ are nonempty.
\end{lemma}
\begin{proof}
  Let $\alpha = \alpha(\hat\eps)$ be the parameter
  of \cite[Corollary 48.1]{kleiner_lott_perelman_notes}.
  Suppose that there is some point $x \in  \M_{t, \hat\eps}$ so that
  $R(x) \diam^2(\M_t) < \alpha$. By the compactness of the space of
  pointed $\kappa$-solutions, it follows that there is an
  upper bound on $(\sup_{\M_t} R) \cdot \diam^2(\M_t)$, depending
  only on $\alpha$.
  If $\epsilon^\prime$ is sufficiently small then we obtain a contradiction.
  Hence we are in case C of
  \cite[Corollary 48.1]{kleiner_lott_perelman_notes}, so there are points
  $m_1, m_2 \in \M_t$ such that
  $M_{t,\hat\eps} \subset
  B(m_1, \alpha R(m_1, t)^{- \frac12}) \cup
  B(m_2, \alpha R(m_2, t)^{- \frac12})$.
  If $\epsilon^\prime$ is sufficiently small
  then a cross-section of the
  $\epsilon^\prime$-neck separates $\M_t$ into two connected components, each
  of which must have a cap region.
  Hence if $\epsilon^\prime$ is sufficiently small then
  $B(m_1, \alpha R(m_1, t)^{- \frac12})$ and
  $B(m_2, \alpha R(m_2, t)^{- \frac12})$ are disjoint.
  As $\M_{t, \hat\eps}$ is closed,
  both $B(m_1, \alpha R(m_1, t)^{- \frac12})$ and
  $B(m_2, \alpha R(m_2, t)^{- \frac12})$
  intersect $\M_t - \M_{t,\hat\eps}$.
\end{proof}

\begin{lemma} \label{pointedgradient}
  Given $\hat\eps > 0$ and
  a compact family ${\mathcal F}$ of pointed $\kappa$-solutions, with
  basepoints at time zero, there
  is some $T = T(\hat\eps, {\mathcal F}) < 0$ such that for each
  $\M \in {\mathcal F}$, there is a point
  $(m,t) \in \M$ with $t \in [-T, 0]$ so that
  $\left( \hat{\M}(-t), m \right)$ is $\hat\eps$-close to a
  pointed gradient shrinking soliton which is a $\kappa$-solution.
\end{lemma}
\begin{proof}
  Suppose that the lemma fails.  Then for each $j \in \Z^+$, there is some
  $\M^j \in {\mathcal F}$ so that for each $(m, t) \in \M^j$ with
  $t \in [-j, 0]$,
  there is no pointed gradient shrinker (which is a $\kappa$-solution)
  that is $\hat\eps$-close to
  $\left( \hat{\M^j}(-t), m \right)$. After passing to a subsequence,
  we can assume that $\lim_{j \rightarrow \infty} \M^j = \M^\infty
  \in {\mathcal F}$.
  From the existence of an asymptotic soliton for $\M^\infty$, there is
  some $(m_\infty, t_\infty) \in \M^\infty$ so that
  $\left( \hat{\M}^\infty(-t_\infty), m_\infty \right)$ is
  $\frac{\hat\eps}{2}$-close to
  a gradient shrinking soliton (which is a $\kappa$-solution).
  Then for large $j$, there is some
  $(m_j, t_\infty) \in \M^j$ so that
  $\left( \hat{\M}^j(-t_\infty), m_j \right)$ is $\hat\eps$-close to the
  gradient shrinking soliton.  This is a contradiction.
\end{proof}

\begin{corollary} \label{corcloseness}
  Let
  ${\mathcal F}$ be a compact family of compact $\kappa$-solutions
  that does not have any constant curvature elements.
  Then for each $\hat\eps > 0$,
  there is some ${\mathcal T} = {\mathcal T}(\hat\eps, {\mathcal F}) < 0$
  such that for each
  $\M \in {\mathcal F}$, there is a point
  $(m,t) \in \M_{[{\mathcal T}, 0]}$ which is the center of an $\hat\eps$-neck.
\end{corollary}
\begin{proof}
  By assumption, there is some $\sigma = \sigma({\mathcal F}) > 0$ so that
  no time-zero slice of an element of ${\mathcal F}$
  is $\sigma$-close to a constant curvature manifold.
  By Lemma \ref{pointedgradient}, for each $\eps^\prime > 0$,
  there is some ${\mathcal T} = {\mathcal T}(\eps^\prime, {\mathcal F}) < 0$
  such that for each
  $\M \in {\mathcal F}$, there is some
  $(m,t) \in \M_{[{\mathcal T}, 0]}$ so that
  $\left( \hat{\mathcal M}(-t), m \right)$
  is $\eps^\prime$-close to a gradient
  shrinking soliton (which is a $\kappa$-solution).
  If $\eps^\prime$ is sufficiently small, in terms
  of $\sigma$, then by the local stability of Ricci flows of constant positive
  curvature, this soliton cannot have constant curvature.  Hence
  it is either a round shrinking cylinder or a $\Z_2$-quotient of a
  round shrinking cylinder. If it is a round shrinking cylinder then
  as long as $\eps^\prime \le \hat\eps$, we are done.
  If it is a $\Z_2$-quotient of a round shrinking cylinder then if
  $\eps^\prime$ is sufficiently small, by moving the basepoint
  we can find a point
  $(m^\prime, t) \in \M$ that is the center of an $\eps$-neck.
\end{proof}

\section{Curvature and volume estimates}
\label{sec_curvature_volume_estimates}

In this section we establish curvature and volume estimates
for singular Ricci flows.  
There are two main results. In
Proposition \ref{prop_r_in_lp}, we show that
$|R|^p$ is integrable on each time slice, for each $p\in(0,1)$.
In Proposition \ref{holder} we give an estimate on how the volume 
$\V(t)$ can change
as a function of $t$. When combined with
part (5) of \cite[Proposition 5.5]{kleiner_lott_singular_ricci_flows}, it shows that 
$\V(t)$ is $\frac{1}{\eta}$-H\"older in $t$,
where $\eta \ge 1$ is the constant in
the estimate
\begin{equation}
  \label{gradest}
  |\nabla R(x, t)|<\eta R(x, t)^{\frac32}, \: \: \:  \: \: \: \:
  \left|\frac{\D R}{\D t} (x, t)\right|<\eta R(x, t)^2,
\end{equation}
for canonical neighborhoods from
\cite[(A.8)]{kleiner_lott_singular_ricci_flows}.

\begin{proposition}
\label{prop_r_in_lp} 
Let $\M$ be a singular Ricci flow. Then for all $p \in (0,1)$ and
$T < \infty$, there is a bound
\begin{equation}
\label{eqn_r_in_lp}
\int_{\M_t} |R|^p \: \dvol_{g(t)} \le
\const (p,T) \vol_{g(0)}(\M_0)
\end{equation}
for all $t \in [0, T]$.
\end{proposition}
\begin{proof}
Before entering into the details, we first give a sketch of the proof. 

Due to the bounds on $\V(t)$ from \cite[Proposition 5.5]{kleiner_lott_singular_ricci_flows}, it suffices to
control the contribution to the left-hand side of 
(\ref{eqn_r_in_lp}) from the points with large scalar curvature.
Such points fall into
three types, according to the geometry of the canonical neighborhoods:
(a) neck points, (b) cap points, and (c) points $p$ whose connected component 
in $\M_t$ is compact and has
diameter   comparable to $R(p)^{-\frac12}$.  If $(p,t)\in \M_t$
is a neck point with worldline $\ga:[0,t]\ra \M$ then thanks to the stability of 
necks going backward in time, the scale-invariant time derivative
$
R^{-1}\frac{\D R}{\D t}
$
will remain very close to the cylindrical value along $\ga$, until $R$
falls down to a value  comparable to $(r(t))^{-2}$.  Combining this
with previous estimates on the Jacobian as in 
\cite[Section 5]{kleiner_lott_singular_ricci_flows}, 
we can bound the contribution
from the neck points
to the left-hand side of (\ref{eqn_r_in_lp}) in terms of the volume
of the corresponding set of points in the 
time zero slice.   To control the 
contribution from points of type (b), we show that it  is 
dominated by that of the neck points.  To control the contribution from the
points of type (c) we use a similar approach. 
We again analyze the geometry going backward in time along
worldlines, except that in this case there are 
three stages : one where the components are nearly 
round, one when they are no longer nearly round but still have
diameter comparable to $R^{-\frac12}$, and one when they have a large
necklike region.

We now start on the proof.
With the notation of the
proof of \cite[Proposition 5.5]{kleiner_lott_singular_ricci_flows}, let $X_3 \subset \M_t$ be the
complement of the set of points in $\M_t$ with a bad worldline.
From \cite[Theorem 7.1]{kleiner_lott_singular_ricci_flows}, it has full measure in $\M_t$.
Given $x \in X_3$, let $\gamma_x : [0,t] \rightarrow \M_{[0, t]}$
be the restriction of its  worldline to the interval $[0,t]$.
Define $J_t(x)$ as in \cite[(5.8)]{kleiner_lott_singular_ricci_flows}, with
$t_1 = 0$.
 That is,
\begin{equation}
J_t(x) = \frac{i_t^* \dvol_{g(t)}}{\dvol_{g(0)}}(x)
\end{equation}
is the pointwise volume distortion of the inclusion map $i_t$ that goes
from (a subset of) the time-zero slice to the time-$t$ slice.
From \cite[(5.9)]{kleiner_lott_singular_ricci_flows}, we have
\begin{equation}
J_{t}(x) = e^{- \: \int_{0}^{t} R(\gamma_x(u)) \: du}.
\end{equation}

Given $T > 0$, 
we consider times $t$ in the range $[0,T]$.
Let $\hat{\epsilon}$, $C_1$, $\overline{R}$ and $\overline{R}'$
be as in \cite[Proposition 5.16]{kleiner_lott_singular_ricci_flows}.
We take $\overline{R} > r(T)^{-2}$.
From \cite[Proposition 5.15]{kleiner_lott_singular_ricci_flows} we can assume that the 
$\hat{\epsilon}$-canonical neighborhood assumption holds on the superlevel set
$\M_{[0,T]}^{> \overline{R}}$ of the scalar curvature function. 
We will further adjust the parameters 
$\hat{\epsilon}$ and $\overline{R}'$ later.

For any $\widehat{R} \ge \overline{R}$,
write
\begin{equation}
\M_t^{> \widehat{R}} = 
\M_{t,neck}^{> \widehat{R}} \cup \M_{t,cap}^{> \widehat{R}} \cup 
\M_{t,closed}^{> \widehat{R}},
\end{equation}
where 
\begin{itemize}
\item 
$\M_{t,neck}^{> \widehat{R}}$ consists of the points in
$\M_t^{> \widehat{R}}$ that
are centers of $\widehat{\epsilon}$-necks,
\item 
$\M_{t,cap}^{> \widehat{R}}$ consists of the points $x
\in \M_t^{> \widehat{R}} - 
\M_{t,neck}^{> \widehat{R}}$ so
that after rescaling by $R(x)$, the pair $(\M_t, x)$ is $\hat{\epsilon}$-close
to a pointed noncompact $\kappa$-solution, and
\item 
$\M_{t,closed}^{> \widehat{R}}
= \M_t^{> \widehat{R}} - 
\left( \M_{t,neck}^{> \widehat{R}} \cup 
\M_{t,cap}^{> \widehat{R}} \right)$.
\end{itemize}

Taking $\widehat{R} = \overline{R}'$, there is a compact set ${\mathcal C}$
of $\kappa$-solutions so that for
$x \in \M_{t,closed}^{> \overline{R}'}$,
after rescaling by $R(x)$ the connected component of $\M_t$ containing $x$ 
is $\hat{\epsilon}$-close to an element of ${\mathcal C}$
(c.f. Step 1 of the proof of \cite[Theorem 7.1]{kleiner_lott_singular_ricci_flows}).
In particular, before rescaling, the diameter of the component is
bounded above by 
$C  R(x)^{- \: \frac12}$ and the scalar curvature on the component satisfies
\begin{equation} \label{oncomp}
C^{-1} R_{av} \le R \le C R_{av},
\end{equation}
for an appropriate constant
$C = C(\hat{\epsilon}) < \infty$,
where $R_{av}$ denote the average scalar curvature on the component. 

By the pointed compactness of the space of normalized $\kappa$-solutions,
and the diameter bound on the caplike regions in normalized pointed
noncompact $\kappa$-solutions,
there is a 
$C' = C'(\hat{\epsilon},\overline{R}') < \infty$ so that 
\begin{equation} \label{capineq}
\int_{\M_{t,cap}^{> \overline{R}'}} |R|^p \: \dvol_{g(t)} \le C'
\int_{\M_{t,neck}^{> \overline{R}'}} |R|^p \: \dvol_{g(t)}.
\end{equation}
Hence we can restrict our attention to 
${\M_{t,neck}^{> \overline{R}'}}$ and
${\M_{t,closed}^{> \overline{R}'}}$.

Consider $x \in \M_{t,neck}^{> \overline{R}'} \cap X_3$. 
With $\delta_{neck}$, $\delta_0$ and $\delta_1$ being
parameters of \cite[Theorem 6.1]{kleiner_lott_singular_ricci_flows}, 
we assume that $\hat{\epsilon} < \frac{\delta_0}{100}$ and
$\delta_1 < \frac{\delta_0}{100}$. Using
\cite[Theorem 6.1]{kleiner_lott_singular_ricci_flows} and the
$\hat{\epsilon}$-canonical neighborhood assumption, there are
$T^{\prime \prime} < T^\prime < 0$ so that 
for $s \in [T^{\prime \prime}, T^\prime]$, the rescaled solution
$\left( \hat{\mathcal M}( - s R(x)^{-1}), 
\gamma_x(t + s R(x)^{-1}) \right)$ is 
$\frac{\delta_0}{10}$-close to $(\cyl, (y_0, -1))$.
By reducing $\hat{\epsilon}$, we can make
$T^{\prime \prime}$ arbitrarily negative.

The gradient bound (\ref{gradest}) gives
\begin{equation} \label{ueqn}
\frac{1}{R(\gamma_x(u))} \le \frac{1}{R(x)} + \eta (t-u).
\end{equation}
as long as $\gamma_x(u)$ stays in a canonical neighborhood.
If 
\begin{equation}
\overline{R}' \ge (1 - \eta T^{\prime \prime}) \overline{R} 
\end{equation}
then for all $u \in [t + T^{\prime \prime} R(x)^{-1}, t]$,
we have $R(\gamma_x(u)) \ge \overline{R}$ and (\ref{ueqn}) holds,
so
\begin{align}
\int_{t + T^\prime R(x)^{-1}}^{t} R(\gamma_x(u)) \: du \: & \ge \:
\int_{t + T^\prime R(x)^{-1}}^{t} 
\frac{1}{\frac{1}{R(x)} + \eta (t-u)} \: du \\
& = \:
\frac{1}{\eta} \log(1 - \eta T^\prime). \notag
\end{align}

For a round shrinking cylinder, the sharp value for $\eta$ in 
(\ref{ueqn}) is $1$. For any $q > 1$, if $\delta_0$ is
sufficiently small then we are ensured that
\begin{align}
\int_{t + T^{\prime \prime} R(x)^{-1}}^{t + T^\prime R(x)^{-1}} 
R(\gamma_x(u)) \: du \: & \ge \:
\int_{t + T^{\prime \prime} R(x)^{-1}}^{t + T^\prime R(x)^{-1}}
\frac{1}{\frac{1}{R(x)} + q (t-u)} \: du \\
& = \:
\frac{1}{q} \log \frac{1 - q T^{\prime \prime}}{1 - q T^\prime}. \notag
\end{align}
In all,
\begin{equation}
e^{- \int_{t + T^{\prime \prime} R(x)^{-1}}^{t} 
R(\gamma_x(u)) \: du} \le (1 - \eta T^\prime)^{- \frac{1}{\eta}}
(1 - q T^\prime)^{\frac{1}{q}}
(1 - q T^{\prime \prime})^{- \frac{1}{q}}.
\end{equation}
Because of the cylindrical approximation,
\begin{equation} \label{2ineq}
\frac12 \le \frac{( 1 - T^{\prime \prime} ) 
R ( \gamma_x(t+T^{\prime \prime} R(x)^{-1} ))}{
 R(x)} \le 2,
\end{equation}
and so there is a constant $C^{\prime \prime} = 
C^{\prime \prime}(q,\eta,T^\prime) < \infty$ such that
for 
very negative
 $T^{\prime \prime}$, we have
\begin{equation}
e^{- \int_{t + T^{\prime \prime} R(x)^{-1}}^{t} 
R(\gamma_x(u)) \: du} \le C^{\prime \prime} 
\left( \frac{R(x)}{R ( \gamma(t+T^{\prime \prime} R(x)^{-1} ))}
\right)^{- \frac{1}{q}}.
\end{equation}

We now replace $t$ by $t+T^{\prime \prime} R(x)^{-1}$ and iterate
the argument. Eventually, there will be a first time $t_x$ when
we can no longer continue the iteration because the curvature 
has gone below $\overline{R}$.
Suppose that there are $N$ such iterations.  Then
\begin{equation}
e^{- \int_{t_x}^{t} 
R(\gamma_x(u)) \: du} \le (C^{\prime \prime})^N  
\left( \frac{R(x)}{\overline{R}} \right)^{- \frac{1}{q}}.
\end{equation}
From (\ref{2ineq}),
\begin{equation}
\left( \frac{1 - T^{\prime \prime}}{2} 
\right)^{N-1} \le \frac{R(x)}{\overline{R}},
\end{equation}
so
\begin{equation}
\label{eqn_t_x_t}
\begin{aligned}
e^{- \int_{t_x}^{t} 
R(\gamma_x(u)) \: du} & \le 
C^{\prime \prime}   
\left( C^{\prime \prime} \right)^{N-1} 
\left( \frac{R(x)}{\overline{R}} \right)^{- \: \frac{1}{q}} \\
& \le C^{\prime \prime}   
\left( \frac{R(x)}{\overline{R}} \right)^{ 
\frac
{\log C^{\prime \prime}
}{
\log \left( \frac{1-T^{\prime \prime}}{2} \right)
}
\: - \: \frac{1}{q}}. 
\end{aligned}
\end{equation}

Put
$p = \frac{1}{q} - \frac
{\log C^{\prime \prime}
}{
\log \left( \frac{1-T^{\prime \prime}}{2} \right)
}$.
By choosing $q$ sufficiently close to $1$ (from above) and $T^{\prime \prime}$
sufficiently negative, we can make $p$ arbitrarily close to $1$
(from below).
Using the lower scalar curvature bound
\cite[Lemma 5.2]{kleiner_lott_singular_ricci_flows},
we have
\begin{equation}
\int_0^{t_x} R(\gamma_x(u)) \: du  \ge -
\int_0^{T} \frac{3}{1+2u} \: du 
= - \frac32 \log(1+2T).
\end{equation}
Then
\begin{equation}
\label{eqn_0_t_x}
e^{- \int_0^{t_x} R(\gamma_x(u)) \: du}  \le
(1+2T)^{\frac32}.
\end{equation}
As
\begin{equation}
  e^{- \int_0^{t} R(\gamma_x(u)) \: du} =
  e^{- \int_0^{t_x} R(\gamma_x(u)) \: du}
  e^{- \int_{t_x}^{t} R(\gamma_x(u)) \: du},
  \end{equation}
by combining (\ref{eqn_t_x_t}) and (\ref{eqn_0_t_x}) we obtain
\begin{equation}
\frac{\dvol_{g(t)}}{\dvol_{g(0)}} (x) = 
J_t(x) \le C^{\prime \prime} (1+2T)^{\frac32} 
\left( \frac{R(x)}{\overline{R}} \right)^{- p}.
\end{equation}
Then 
\begin{align} \label{newcomb1}
\int_{\M_{t,neck}^{> \overline{R}'}} R^{p} \: \dvol_{g(t)} & \le
C^{\prime \prime} \overline{R}^p (1+ 2T)^{\frac32}
\int_{\M_{t,neck}^{> \overline{R}'}} \dvol_{g(0)} \\
& \le C^{\prime \prime} \overline{R}^p  (1+ 2T)^{\frac32}
\vol_{g(0)}(\M_0). \notag
\end{align}
This finishes the discussion of the neck points.

Let $\overline{R}'' > \overline{R}'$ be a new parameter.
Given $\sigma > 0$ small, let 
${\mathcal M}_{t,round}$ be the connected components of
${\mathcal M}_t$ that intersect
${\mathcal M}_{t,closed}^{> \overline{R}''}$ and
are $\sigma$-close to a
constant curvature metric, and let 
${\mathcal M}_{t,nonround}$ be the other connected 
components of $\M_t$ that intersect 
${\mathcal M}_{t,closed}^{> \overline{R}''}$.
Using \cite[Proposition 5.17]{kleiner_lott_singular_ricci_flows}, 
a connected component ${\mathcal N}_t$ in
${\mathcal M}_{t}$ determines a connected component
${\mathcal N}_{t^\prime}$ in
${\mathcal M}_{t^\prime}$ for all $t^\prime \le t$.

Let ${\mathcal N}_t$ be a component in 
${\mathcal M}_{t,nonround}$.
From (\ref{oncomp}), we have
$R_{av} \ge C^{-1} \overline{R}''$.
Using the compactness of the space of approximating $\kappa$-solutions,
we can apply Lemma \ref{lemcloseness} and Corollary \ref{corcloseness}.
Then for ${\epsilon}^\prime$ small and
${\mathcal T} = {\mathcal T}({\epsilon}^\prime) < \infty$,
there is some $t^\prime \in [t, t - 10 C {\mathcal T} 
R_{av}^{-1}]$ so that ${\mathcal N}_{t^\prime}$ consists of centers of 
${\epsilon}^\prime$-necks and two caps. From
(\ref{ueqn}), if $x \in {\mathcal N}_t$  and
\begin{equation} \label{needed}
R \Big|_{\gamma_x([t^\prime, t])} \ge \overline{R}
\end{equation}
then
\begin{equation}
\frac{1}{R(\gamma_x(t^\prime))} \le \frac{C}{R_{av}} + 10 C \eta {\mathcal T}
R_{av}^{-1},
\end{equation}
so
\begin{equation} \label{easy}
R(\gamma_x(t^\prime)) \ge \frac{R_{av}}{C(1+10 \eta {\mathcal T})} \ge
\frac{\overline{R}''}{C^2 (1+10 \eta {\mathcal T})}.
\end{equation}
If $\overline{R}'' > C^2 (1 + 10 \eta {\mathcal T}) \overline{R}'$
then (\ref{needed}) holds and from (\ref{easy}),
${\mathcal N}_{t^\prime} \subset {\mathcal M}_{t^\prime}^{>
\overline{R}'}$. From (\ref{oncomp}) and (\ref{easy}), we also have
\begin{equation}
R(x) \le C R_{av} \le C^2 (1 + 10 \eta {\mathcal T}) R(\gamma_x(t^\prime)).
\end{equation}
Since
the volume element at $\gamma_u(x)$ is nonincreasing as a function of $u \in
[t^\prime, t]$, we obtain
\begin{equation}
\int_{{\mathcal N}_t} R \: \dvol_{g(t)} \le C^2 (1 + 10 \eta {\mathcal T})
\int_{{\mathcal N}_{t^\prime}} R \: \dvol_{g(t^\prime)}.
\end{equation}
We now apply the argument starting with (\ref{capineq}) to
${\mathcal N}_{t^\prime}$. Taking
$\frac{\overline{R}'}{\overline{R}}$ large
compared to $\frac{\overline{R}''}{\overline{R}'}$, in order to ensure
many iterations in the earlier-neck argument,
we get a bound 
\begin{equation}
\int_{{\mathcal N}_t} |R|^p \: \dvol_{g(t)} \le \const (p,T) 
\vol_{g(0)}({\mathcal N}_t).
\end{equation}
This takes care of the components in ${\mathcal M}_{t,nonround}$.

Let ${\mathcal N}_t$ be a component of ${\mathcal M}_t$ in
${\mathcal M}_{t,round}$.
Let $\tau$ be the infimum of the $u$'s 
so that for all $t^\prime \in [u, t]$,
the metric on ${\mathcal N}_{t^\prime}$ is $\sigma$-close to a
constant curvature metric. 
For a Ricci flow solution with time slices of
constant positive curvature, $R$ is strictly increasing
along forward worldlines but
$\int R \: \dvol$ is strictly decreasing in $t$.  Hence if $\sigma$
is sufficiently small then we are ensured that
\begin{equation}
\int_{{\mathcal N}_t} R \: \dvol_{g(t)} \le 
\int_{{\mathcal N}_{\tau}} R \: \dvol_{g(\tau)}.
\end{equation}

If $N_{\tau}$ has a point with scalar curvature
at most $\overline{R}''$ and $\sigma$ is small then
\begin{equation}
\int_{{\mathcal N}_{\tau}} R \: \dvol \le
2 \overline{R}'' \vol_{g(\tau)} 
\left( {\mathcal N}_{\tau} \right) \le
2 \overline{R}'' (1+2T)^{\frac32}  \vol_{g(0)}
\left( {\mathcal N}_t \right).
\end{equation}
If, on the other hand, $N_{\tau} \subset \M_{\tau}^{> \overline{R}''}$ then
we can apply the preceding argument for $\M_{t,nonround}$, replacing
$t$ by $\tau$. The conclusion is that
\begin{equation} \label{newcomb2}
\int_{{\mathcal M}^{> \overline{R}''}_{t,closed}} 
|R|^p \: \dvol_{g(t)} \le \const (p,T) 
\vol_{g(0)}({\mathcal M}^{> \overline{R}''}_{t,closed}).
\end{equation}

Since 
\begin{equation} \label{newcomb3}
\int_{\M_{t,neck}^{> \overline{R}''}} R^{p} \: \dvol_{g(t)} \le
\int_{\M_{t,neck}^{> \overline{R}'}} R^{p} \: \dvol_{g(t)},
\end{equation}
\begin{equation} \label{newcomb3.5}
\int_{\M_{t,cap}^{> \overline{R}''}} R^{p} \: \dvol_{g(t)} \le
\int_{\M_{t,cap}^{> \overline{R}'}} R^{p} \: \dvol_{g(t)}
\end{equation}
and
\begin{equation} \label{newcomb4}
\int_{\M_{t}^{\le \overline{R}''}} |R|^p \: \dvol_{g(t)} \le
\left( \overline{R}'' \right)^p \: (1+2T)^{\frac32}
\vol_{g(0)} \left( \M_{t}^{\le \overline{R}''} \right), 
\end{equation}
the proposition follows from
(\ref{capineq}), 
(\ref{newcomb1}), (\ref{newcomb2}), (\ref{newcomb3}), 
(\ref{newcomb3.5}) and 
(\ref{newcomb4}).
\end{proof}

\begin{proposition} \label{holder}
Let $\M$ be a singular Ricci flow.
Let $\eta$ be the constant from (\ref{gradest}). We can
assume that $\eta \ge 1$. 
Then
whenever $0 \le t_1 \le t_2 < \infty$ satisfies
$t_2 - t_1 < \frac{1}{\eta} r(t_2)^2$ and $t_1 > \frac{1}{100 \eta}$, we have
\begin{align} \label{volest}
& \V(t_2) - \V(t_1) \ge \\
& - \eta^{\frac{1}{\eta}}
\left( 2  \int_{\M_{t_1}}
|R|^{\frac{1}{\eta}} \: \dvol_{g(t_1)}
+ r(t_2)^{- \: \frac{2}{\eta}} \V(t_1) 
\right) (t_2 - t_1)^{\frac{1}{\eta}}  \ge \notag \\
& - 5 \eta^{\frac{1}{\eta}}
r(t_2)^{- \: \frac{2}{\eta}} (1+2t_1)^{\frac32} \V(0) 
\cdot (t_2 - t_1)^{\frac{1}{\eta}}. \notag
\end{align}
\end{proposition}
\begin{proof}
Let $X_1 \subset \M_{t_1}$ be the set of points $x\in \M_{t_1}$
whose worldline $\gamma_x$
extends forward to time $t_2$ and
let $X_2 \subset \M_{t_1}$ be the points $x$
whose worldline $\gamma_x$
does not extend forward to time $t_2$. 
Put
\begin{equation}
X_1^{\prime} = 
\left\{ x \in X_1 \: : \: R(x) > \frac{1}{\eta(t_2 - t_1)} \right\},
\end{equation}

\begin{equation}
X_1^{\prime \prime} = 
\left\{ x \in X_1 \: : \: r(t_2)^{-2} < R(x) \le
\frac{1}{\eta(t_2 - t_1)} \right\}
\end{equation}
and
\begin{equation}
X_1^{\prime \prime \prime} = \left\{ 
x \in X_1 \: : \: R(x) \le r(t_2)^{-2} \right\}.
\end{equation}
Then
\begin{align} \label{addup1}
\vol(\M_{t_2}) - \vol(\M_{t_1}) \ge &
\vol_{t_2} \left( X_1^{\prime} \right) -
\vol_{t_1} \left( X_1^{\prime} \right) + \\
& \vol_{t_2} \left( X_1^{\prime \prime} \right) -
\vol_{t_1} \left( X_1^{\prime \prime} \right) + \notag \\ 
& \vol_{t_2} \left( X_1^{\prime \prime \prime} \right) -
\vol_{t_1} \left( X_1^{\prime \prime \prime} \right) -
\vol_{t_1} \left( X_2 \right) \notag \\
\ge &
\vol_{t_2} \left( X_1^{\prime \prime} \right) -
\vol_{t_1} \left( X_1^{\prime \prime} \right) + 
\vol_{t_2} \left( X_1^{\prime \prime \prime} \right) - \notag \\
& \vol_{t_1} \left( X_1^{\prime \prime \prime} \right) -
\vol_{t_1} \left( X_2 \right) - \vol_{t_1} \left( X_1^{\prime} \right). \notag
\end{align}

\bigskip

Suppose that $x \in X_2$.

\begin{lemma}
Let $[t_1, t_x)$ be the domain of the forward extension of
$\gamma_x$, with $t_x < t_2$. 
For all $u \in [t_1, t_x)$, we have
\begin{equation} \label{crawl}
R(\gamma_x(u)) \ge \frac{1}{\eta(t_x - u)}.
\end{equation}
\end{lemma}
\begin{proof}
If the lemma is not true, put 
\begin{equation}
u^\prime = \sup \left\{u \in [t_1, t_x) \: : \: 
R(\gamma_x(u)) < \frac{1}{\eta(t_x - u)} \right\}.
\end{equation}
Then $u^\prime > t_1$. From the
gradient estimate (\ref{gradest}) and the fact that
$\lim_{u \ra t_x} R(\gamma_x(u)) = \infty$, we know that
$u^\prime < t_x$. Whenever $u \ge u^\prime$, we have
\begin{equation}
R(\gamma_x(u)) \ge \frac{1}{\eta(u_x - u^\prime)} \ge
\frac{1}{\eta(t_2  - t_1)} > r(t_2)^{-2},
\end{equation} 
so there is some $\mu > 0$ so that
the gradient estimate (\ref{gradest}) holds on the
interval $(u^\prime - \mu, t_x)$.
This implies that
(\ref{crawl}) holds for all $u \in (u^\prime - \mu, t_x)$,
which contradicts the definition of $u^\prime$.
This proves the lemma.
\end{proof}

Hence
\begin{equation}
(X_2 \cup X_1^\prime) \subset 
\left\{x \in \M_{t_1} \: : \: 
R(x) \ge \frac{1}{\eta(t_2 - t_1)} \right\}
\end{equation}
and
\begin{align} \label{addup2}
\vol_{t_1}(X_2) + \vol_{t_1} (X_1^\prime) 
& \le \vol \left\{x \in \M_{t_1} \: : \: 
R(x) \ge \frac{1}{\eta(t_2 - t_1)} \right\} \\
& \le
\eta^{\frac{1}{\eta}} (t_2 - t_1)^{\frac{1}{\eta}}
\int_{\M_{t_1}} |R|^{\frac{1}{\eta}} \: \dvol_{g(t_1)}, \notag
\end{align}
since $\eta^{\frac{1}{\eta}}(t_2-t_1)^\frac{1}{\eta}|R|^\frac{1}{\eta}\geq 1$
on the set $\{x\in \M_{t_1}: R(x)\geq \frac{1}{\eta(t_2-t_1)}\}$.

Suppose now that $x \in X_1^{\prime \prime}$.
\begin{lemma}
For all $u \in [t_1, t_2]$, we have
\begin{equation} \label{estpp}
R(\gamma_x(u)) \le \frac{1}{\frac{1}{R(x)} - \eta(u-t_1)} < \infty.
\end{equation}
\end{lemma}
\begin{proof}
If the lemma is not true, put
\begin{equation}
u^{\prime \prime} = \inf \left\{u \in [t_1, t_2] \: : \:
R(\gamma_x(u)) > \frac{1}{\frac{1}{R(x)} - \eta(u-t_1)} \right\}.
\end{equation}
Then $u^{\prime \prime} < t_2$ and
the gradient estimate (\ref{gradest}) implies that
$u^{\prime \prime} > t_1$. 
Now
\begin{equation} \label{Reqal}
R(\gamma_x(u^{\prime \prime})) = \frac{1}{\frac{1}{R(x)} - 
\eta(u^{\prime \prime}-t_1)} 
> R(x) > r(t_2)^{-2}.
\end{equation}
Hence there is some $\mu > 0$ so that
$R(\gamma_x(u)) \ge r(t_2)^{-2}$ for $u \in [u^{\prime \prime},
u^{\prime \prime} + \mu]$.
If $R(\gamma_x(u)) \ge r(t_2)^{-2}$ for all 
$u \in [t_1, u^{\prime \prime}]$
then (\ref{gradest}) implies that (\ref{estpp}) holds for
$u \in [t_1, u^{\prime \prime} + \mu]$, which contradicts the definition of
$u^{\prime \prime}$. 
On the other hand, if it is not true that 
$R(\gamma_x(u)) \ge r(t_2)^{-2}$ for all 
$u \in [t_1, u^{\prime \prime}]$, put
\begin{equation}
v^{\prime \prime} = 
\sup \left\{u \in [t_1, u^{\prime \prime}] \: : \: R(\gamma_x(u)) <
r(t_2)^{-2} \right\}.
\end{equation}
Then $v^{\prime \prime} > t_1$ and 
$R(\gamma_x(v^{\prime \prime})) =
r(t_2)^{-2}$. Equation
(\ref{gradest}) implies that
\begin{equation}
R(\gamma_x(u^{\prime \prime})) \le 
\frac{1}{r(t_2)^2 - \eta (u^{\prime \prime} - v^{\prime \prime})} <
\frac{1}{\frac{1}{R(x)} - \eta(u^{\prime \prime} -t_1)},
\end{equation}
which contradicts (\ref{Reqal}). This proves the lemma.
\end{proof}

Hence if $x \in X_1^{\prime \prime}$ then
\begin{align}
\int_{t_1}^{t_2} R(\gamma_x(u)) \: du & \le 
\int_{t_1}^{t_2} \frac{1}{\frac{1}{R(x)} - \eta (u-t_1)} \: du \\
& =
- \: \frac{1}{\eta} \log \left( 1 - \eta R(x) \cdot (t_2-t_1) \right), \notag
\end{align}
so
\begin{equation}
\frac{\dvol_{g(t_2)}}{\dvol_{g(t_1)}} (x) = 
J_{t_2}(x) \ge \left( 1 - \eta R(x) \cdot (t_2-t_1) \right)^{\frac{1}{\eta}}.
\end{equation}
Thus
\begin{align}
& \vol_{t_2} \left( X_1^{\prime \prime} \right) -
\vol_{t_1} \left( X_1^{\prime \prime} \right) \ge \\
& \int_{X_1^{\prime \prime}} \left(
\left( 1 - \eta  R \cdot (t_2-t_1) \right)^{\frac{1}{\eta}} - 1
\right) \: \dvol_{g(t_1)}. \notag
\end{align}

Since $\eta \ge 1$, if $z \in [0,1]$ then
$\left( z^{\frac{1}{\eta}} \right)^{\eta} + 
\left( 1-z^{\frac{1}{\eta}} \right)^\eta \le 1$, so
\begin{equation} \label{trivial}
\left(1-z \right)^{\frac{1}{\eta}} - 1 \ge - z^{\frac{1}{\eta}}.
\end{equation}
Then
\begin{equation} \label{addup3}
\vol_{t_2} \left( X_1^{\prime \prime} \right) -
\vol_{t_1} \left( X_1^{\prime \prime} \right) \ge
- \eta^{\frac{1}{\eta}} (t_2-t_1)^{\frac{1}{\eta}}
\int_{X_1^{\prime \prime}} R^{\frac{1}{\eta}} \: \dvol_{g(t_1)}.
\end{equation}

Now suppose that $x \in X_1^{\prime \prime \prime}$.
\begin{lemma}
For all $u \in [t_1, t_2]$, we have
\begin{equation} \label{estppp}
R(\gamma_x(u)) \le \frac{1}{r(t_2)^2 - \eta(u-t_1)} < \infty.
\end{equation}
\end{lemma}
\begin{proof}
If the lemma is not true, put
\begin{equation}
u^{\prime \prime \prime} = \inf \left\{u \in [t_1, t_2] \: : \:
R(\gamma_x(u)) > \frac{1}{r(t_2)^2 - \eta(u-t_1)} \right\}.
\end{equation}
Then $u^{\prime \prime \prime} < t_2$.
If $R(x) < r(t_2)^{-2}$ then clearly $u^{\prime \prime \prime} > t_1$.
If $R(x) = r(t_2)^{-2}$ then
since $r(t_1) > r(t_2)$,
there is some $\nu > 0$ so that $R(\gamma_x(u)) > r(u)^{-2}$ for
$u \in [t_1, t_1 + \nu]$; then
(\ref{gradest}) gives the validity of (\ref{estppp}) for 
$u \in [t_1, t_1 + \nu]$, which
implies that 
$u^{\prime \prime \prime} > t_1$. In either case,
$t_1 < u^{\prime \prime \prime} < t_2$.
Now
\begin{equation} \label{Reqal2}
R(\gamma_x(u^{\prime \prime \prime})) = \frac{1}{r(t_2)^2 - 
\eta(u^{\prime \prime \prime}-t_1)} 
> r(t_2)^{-2}.
\end{equation}
Hence there is some $\mu > 0$ so that
$R(\gamma_x(u)) \ge r(t_2)^{-2}$ for $u \in [u^{\prime \prime \prime},
u^{\prime \prime \prime} + \mu]$.
If $R(\gamma_x(u)) \ge r(t_2)^{-2}$ for all 
$u \in [t_1, u^{\prime \prime \prime}]$
then (\ref{gradest}) implies that (\ref{estppp}) holds for
$u \in [t_1, u^{\prime \prime \prime} + \mu]$, 
which contradicts the definition of
$u^{\prime \prime \prime}$. 
On the other hand, if it is not true that 
$R(\gamma_x(u)) \ge r(t_2)^{-2}$ for all 
$u \in [t_1, u^{\prime \prime \prime}]$, put
\begin{equation}
v^{\prime \prime \prime} = 
\sup \left\{u \in [t_1, u^{\prime \prime \prime}] \: : \: R(\gamma_x(u)) <
r(t_2)^{-2} \right\}.
\end{equation}
Then $v^{\prime \prime \prime} > t_1$ and
$R(\gamma_x(v^{\prime \prime \prime})) = r(t_2)^{-2}$.
The gradient estimate
(\ref{gradest}) implies that
\begin{equation}
R(\gamma_x(u^{\prime \prime \prime}) \le 
\frac{1}{r(t_2)^2 - \eta (u^{\prime \prime \prime} - 
v^{\prime \prime \prime})} <
\frac{1}{r(t_2)^2 - \eta(u^{\prime \prime \prime} -t_1)},
\end{equation}
which contradicts (\ref{Reqal2}). This proves the lemma.
\end{proof}

Hence if $x \in X_1^{\prime \prime \prime}$ then
\begin{align}
\int_{t_1}^{t_2} R(\gamma_x(u)) \: du & \le 
\int_{t_1}^{t_2} \frac{1}{r(t_2)^2 - \eta (u-t_1)} \: du \\
& =
- \: \frac{1}{\eta} \log \left( 1 - \eta r(t_2)^{-2} \cdot (t_2-t_1) \right),
\notag
\end{align}
so
\begin{equation}
\frac{\dvol_{g(t_2)}}{\dvol_{g(t_1)}} (x) = 
J_{t_2}(x) \ge \left( 1 - \eta r(t_2)^{-2} 
\cdot (t_2-t_1) \right)^{\frac{1}{\eta}}.
\end{equation}
Thus
\begin{align}
& \vol_{t_2} \left( X_1^{\prime \prime \prime} \right) -
\vol_{t_1} \left( X_1^{\prime \prime \prime} \right) \ge \\
& \int_{X_1^{\prime \prime \prime}} \left(
\left( 1 - \eta  r(t_2)^{-2} \cdot (t_2-t_1) \right)^{\frac{1}{\eta}} - 1
\right) \: \dvol_{g(t_1)}. \notag
\end{align}
Since $\eta r(t_2)^{-2} \cdot (t_2 - t_1) \in [0, 1]$, we can apply
(\ref{trivial})
to conclude that
\begin{align} \label{addup4}
\vol_{t_2} \left( X_1^{\prime \prime \prime} \right) -
\vol_{t_1} \left( X_1^{\prime \prime \prime} \right)  \ge 
& - \eta^{\frac{1}{\eta}} r(t_2)^{- \: \frac{2}{\eta}} \cdot 
(t_2-t_1)^{\frac{1}{\eta}}
\vol_{t_1} \left( X_1^{\prime \prime \prime} \right) \\
\ge &  - \eta^{\frac{1}{\eta}} r(t_2)^{- \: \frac{2}{\eta}} \V(t_1) \cdot 
(t_2-t_1)^{\frac{1}{\eta}}. \notag
\end{align}

Combining (\ref{addup1}), (\ref{addup2}), (\ref{addup3}) and (\ref{addup4})
gives (\ref{volest}).
\end{proof}

\section{Asymptotic conditions} \label{asympcon}

In this section we show that the {\it a priori} assumptions
in Definition \ref{srf} are really conditions
on the spacetime near infinity.
That is,
given $\epsilon > 0 $ and a 
decreasing function $r : [0, \infty) \ra (0, \infty)$, there are
decreasing functions $\kappa^\prime = \kappa^\prime(\epsilon) :
[0, \infty) \rightarrow (0, \infty)$ and $r^\prime = r^\prime(\epsilon,r) :
[0, \infty) \rightarrow (0, \infty)$ with the following property.
Let $\M$ be a Ricci flow spacetime with normalized initial condition,
on which condition (a) of Definition \ref{srf} holds.
Suppose that for each
$T \ge 0$ there is a compact subset of
$\M_{\le T}$ so that condition (b), and the $r$-canonical neighborhood
assumption of condition (c), hold on the part of
$\M_{\le T}$ outside of the compact
subset. Then $\M$ satisfies Definition \ref{srf}
globally with parameters $\epsilon$, $\kappa^\prime$ and $r^\prime$.

If $\M$ is a Ricci flow spacetime and $m_0 \in \M$, put $t_0 = \t(m_0)$.
We define Perelman's
$l$-function using curves emanating backward from $m_0$, as in 
\cite[Section 15]{kleiner_lott_perelman_notes}.
That is, given $m \in \M$ with $\t(m) < t_0$, consider a
time-preserving map $\gamma : [\t(m), t_0] \rightarrow \M$ from $m$ to $m_0$.
We reparametrize $[\t(m), t_0]$ by $\tau(t) = t_0 - t$. Then
\begin{equation}
{\mathcal L}(\gamma) = \int_0^{t_0 - \t(m)} 
\sqrt{\tau} \left( R(\gamma(\tau)) +
|\dot{\gamma}(\tau)|^2 \right) \: d\tau,
\end{equation}
where $\dot{\gamma}$ is the spatial projection of the velocity
vector of $\gamma$ and 
$|\dot{\gamma}(\tau)|$ is computed using the metric on 
$\M_{t_0 - \tau}$ at $\gamma(\tau)$. 
Let $L(m)$ be the infimal ${\mathcal L}$-length of such curves $\gamma$.
The reduced length is
\begin{equation}
l(m) = \frac{L(m)}{2 \sqrt{t_0 - \t(m)}}.
\end{equation}

\begin{proposition} \label{box}
Given $\overline{\Delta r}, \overline{\Delta t} > 0$
there are 
$\Delta r = \Delta r(m_0) < \overline{\Delta r}$ and
$\Delta t = \Delta t(m_0) < \overline{\Delta t}$ with the
following property.  For any $m \in \M$ with $\t(m) < t_0 - \Delta t$, 
let $d_{qp}(m, P(m_0, \Delta r, - \Delta t))$ denote the
$g_{\M}^{qp}$-distance from $m$ to the set $P(m_0, \Delta r, - \Delta t)$.
Then
\begin{equation} \label{leqn}
L(m) \ge 
\min \left( \frac{(\Delta r)^2}{4 \sqrt{\Delta t}},
\frac{\sqrt{\Delta t}}{10} 
 \: d_{qp}(m, P(m_0, \Delta r, - \Delta t)) \right) - 
\frac83 t_0^{\frac32}.
\end{equation} 
\end{proposition}
\begin{proof}
With a slight variation on Perelman's definition
\cite[Definition 79.1]{kleiner_lott_perelman_notes}, we put
\begin{equation}
{\mathcal L}_+(\gamma) = \int_0^{t_0 - \t(m)} 
\sqrt{\tau} \left( R_+(\gamma(\tau)) +
|\dot{\gamma}(\tau)|^2 \right) \: d\tau,
\end{equation}
where $R_+(m) = \max(R(m), 1)$. We define $L_+(m)$ using
${\mathcal L}_+(\gamma)$ instead of ${\mathcal L}(\gamma)$.
Applying the lower curvature bound
\cite[(5.3)]{kleiner_lott_singular_ricci_flows} (with $C = n=3$), we know that $R \ge -3$ and so
\begin{align}
{\mathcal L}(\gamma) - {\mathcal L}_+(\gamma) = &
\int_0^{t_0 - \t(m)} 
\sqrt{\tau} \left( R(\gamma(\tau)) -
 R_+(\gamma(\tau)) \right) \: d\tau \\
\ge & - \int_0^{t_0 - \t(m)} 
\sqrt{\tau} \cdot 4 \: d\tau \ge - \frac83 t_0^{\frac32}. \notag
\end{align}
Hence it suffices to estimate ${\mathcal L}_+(\gamma)$ from below.

Given numbers $\Delta r, \Delta t > 0$, if $\t(m) < t_0 - \Delta t$ then
\begin{equation} \label{pareu}
{\mathcal L}_+(\gamma) \ge \int_0^{\Delta t} \sqrt{\tau} 
|\dot{\gamma}(\tau)|^2 \: d\tau = 
\frac12 \int_0^{\sqrt{\Delta t}}
\left|\frac{d\gamma}{ds} \right|^2 \: ds.
\end{equation}
Suppose first that $\gamma$ leaves $m_0$ and exits
$P(m_0, \Delta r, - \Delta t)$ at some time $t \in (t_0 - \Delta t, t_0)$.
If the parabolic ball were Euclidean then we could say from (\ref{pareu}) that
${\mathcal L}_+(\gamma) \ge \frac12 \frac{(\Delta r)^2}{\sqrt{\Delta t}}$. 
If $\Delta r$ and $\Delta t$ are small enough,
depending on $m_0$, then we can still say that
$P(m_0, \Delta r, - \Delta t)$ is unscathed and 
${\mathcal L}_+(\gamma) \ge \frac14 \frac{(\Delta r)^2}{\sqrt{\Delta t}}$.

Given such values of $\Delta r$ and $\Delta t$, 
suppose now that $\gamma$ does not exit
$P(m_0, \Delta r, - \Delta t)$ in the time interval $(t_0 - \Delta t, t_0)$.
Then $\gamma(t_0 - \Delta t) \in P(m_0, \Delta r, - \Delta t)$.
Now
\begin{equation}
{\mathcal L}_+(\gamma) \ge \sqrt{\Delta t} \int_{\Delta t}^{t_0 - \t(m)} 
\left( R_+(\gamma(\tau)) +
|\dot{\gamma}(\tau)|^2 \right) \: d\tau,
\end{equation}

Since $R \ge -3$, it follows that along $\gamma$, we have
\begin{equation}
\sqrt{1+R^2} \le 10 R_+.
\end{equation}
Then 
\begin{align}
\sqrt{1+R^2} |\dot{\gamma}|^2 + 1 + R^2 \le
& 10 R_+ |\dot{\gamma}|^2 + 100 R_+^2 \\
\le &
100 \left( |\dot{\gamma}|^4 + 2 R_+ |\dot{\gamma}|^2 + R_+^2 \right), \notag 
\end{align}
so 
\begin{equation}
\sqrt{\sqrt{1+R^2} |\dot{\gamma}|^2 + 1 + R^2} \le
10 \left( |\dot{\gamma}|^2 + R_+ \right).
\end{equation}
Thus 
\begin{align}
{\mathcal L}_+(\gamma) \ge & \frac{\sqrt{\Delta t}}{10}
\int_{\Delta t}^{t_0 - \t(m)}
\left| \frac{d\gamma}{d\tau} \right|_{g_{\M}^{qp}} \: d\tau \\
\ge & \frac{\sqrt{\Delta t}}{10} 
d_{qp}(m, P(m_0, \Delta r, - \Delta t)). \notag
\end{align}
The proposition follows.
\end{proof}

\begin{corollary} \label{boxx}
Suppose that 
$(\M_{\le t_0}, g_{\M}^{qp})$ is complete away from the time-zero slice
and the time-$t_0$ slice.
Given $t^\prime < t_0$, the restriction of $L$ to
$\M_{\le t^\prime}$ is proper and bounded below.
\end{corollary}
\begin{proof}
From (\ref{leqn}), the function $L$ is bounded below on $\M_{\le t^\prime}$.
Suppose that it is not proper.  Then for some $C < \infty$, there
is a sequence $\{m_i\}_{i=1}^\infty$ in $\M_{\le t^\prime}$ going to
infinity with $L(m_i) < C$ for all $i$. We can choose
$\Delta r,\Delta t > 0$ with $\Delta t < t_0 - t^\prime$ and
\begin{equation}
\frac{(\Delta r)^2}{4 \sqrt{\Delta t}} - \frac{8}{3} t_0^{\frac32}
\ge C.
\end{equation}
By the completeness of $g_{\M}^{qp}$,
we have $\lim_{i \rightarrow \infty}
d_{qp}(m_i, P(m_0, \Delta r, - \Delta t)) = \infty$.
Then (\ref{leqn}) gives a contradiction..
\end{proof}

It is not hard to see that $L$ is continuous on $\M_{< t_0}$.
From the proof of Proposition \ref{box}, given $m \in \M_{< t_0}$
and $K < \infty$, the time-preserving curves $\gamma \: : \:
[\t(m), t_0] \rightarrow \M$ from $m$ to $m_0$ with 
${\mathcal L}(\gamma) < K$ lie in a compact subset of $\M$.
From standard arguments \cite[Section 17]{kleiner_lott_perelman_notes},
it follows that there is an ${\mathcal L}$-minimizer from
$m_0$ to $m$. 

Since $L$ is bounded below and time-slices have finite volume
from \cite[Corollary 7.7]{kleiner_lott_singular_ricci_flows}, the reduced volume
$\tilde{V}(\tau) = \tau^{- \: \frac32} \int_{\M_{t_0 - \tau}}
e^{-l} \dvol$ exists.  The results of 
\cite[Sections 17-29]{kleiner_lott_perelman_notes} go through in our setting.
In particular, $\tilde{V}(\tau)$ is nonincreasing in $\tau$.

\begin{proposition} \label{smalll}
Suppose that 
$(\M_{\le t_0}, g_{\M}^{qp})$ is complete away from the time-zero slice
and the time-$t_0$ slice.
For every $t \in [0, t_0)$, there is some $m \in \M_t$ with
$l(m) \le \frac32$.
\end{proposition}
\begin{proof}
Putting $\overline{L} = 2  \sqrt{t_0 - \t} \: L$, we have
\begin{equation}
\partial_{\t} (- \overline{L} + 6(t_0 - \t)) \le 
 \triangle (- \overline{L} + 6(t_0 - \t))
\end{equation}
in the barrier sense \cite[Section 24]{kleiner_lott_perelman_notes}.
From Corollary \ref{boxx}, for each $\widetilde{t}^\prime \in 
[0, t_0)$, the function $- \overline{L} + 6(t_0 - \t)$ is 
proper and bounded above on $\M_{\le \widetilde{t}^\prime}$.
In particular, for each $t \in [0, t_0)$, the
maximum of  $- \overline{L} + 6(t_0 - t)$ exists on $\M_t$. We want
to show that the maximum is nonnegative.
By way of contradiction,
suppose that for some $\widetilde{t} \in [0, t_0)$ and some $\alpha < 0$,
we have $- \overline{L}(m) + 6(t_0 - \widetilde{t}) \le \alpha$ for
all $m \in \M_{\widetilde{t}}$. Given $\widetilde{t}^\prime \in 
(\widetilde{t}, t_0)$, we can apply \cite[Lemma 5.1]{kleiner_lott_singular_ricci_flows}
on the interval $[\widetilde{t}, \widetilde{t}^\prime]$ to conclude
that $- \overline{L}(m) + 6(t_0 - \widetilde{t}^\prime) \le \alpha$
for all $m \in \M_{\widetilde{t}^\prime}$. 
However,
along the worldline $\gamma$ going through $m_0$, for small $\tau > 0$ we have
\begin{equation} \label{quadr}
\overline{L}(\gamma(t_0 - \tau)) \le \const \tau^2.
\end{equation}
Then for small $\tau$, we have
$- \overline{L}(\gamma(t_0 - \tau)) + 6\tau > 0$. Taking
$\widetilde{t}^\prime = t_0 - \tau$
gives a contradiction and proves the proposition. 
\end{proof}

In his first Ricci flow paper, Perelman showed that there is a
decreasing function $\kappa^\prime \: : \:
[0, \infty) \rightarrow (0, \infty)$ with the property that
if $\M$ is a smooth Ricci flow solution, with normalized
initial conditions, then $\M$ is $\kappa^\prime$-noncollapsed at
scales less than $\epsilon$ \cite[Theorem 26.2]{kleiner_lott_perelman_notes}. 

\begin{proposition} \label{newkappa}
Let $\M$ be a Ricci flow spacetime
 with normalized initial condition. 
Given $t^\prime > 0$, 
suppose that
$(\M_{\le t^\prime}, g_{\M}^{qp})$ is complete away from the time-zero slice
and the time-$t^\prime$ slices.
Then $\M_{\le t^\prime}$ 
is $\kappa^\prime$-noncollapsed at scales less
than $\epsilon$.
\end{proposition}
\begin{proof}
The proof is along the lines of that of
\cite[Theorem 26.2]{kleiner_lott_perelman_notes}. 
We can assume that $t^\prime > \frac{1}{100}$.
To prove $\kappa^\prime$-noncollapsing near 
$m_0 \in \M_{\le t^\prime}$,
we consider ${\mathcal L}$-curves emanating backward
in time from $m_0$ to a fixed time slice $\M_{\overline{t}}$, say with
$\overline{t} = \frac{1}{100}$. By Proposition \ref{smalll}, there 
is some $m \in \M_{\overline{t}}$ with $l(m) \le \frac32$.
Using the bounded geometry near $m$ and the monotonicity of 
$\tilde{V}$, the $\kappa^\prime$-noncollapsing follows as in
\cite[Pf. of Theorem 26.2]{kleiner_lott_perelman_notes}.
\end{proof}

We now show that the conditions in Definition \ref{srf}, to define
a singular Ricci flow, are actually asymptotic in nature.

\begin{proposition} \label{asympprop}
Given $\epsilon > 0$, $t^\prime < \infty$ and a decreasing function 
$r \: : \: [0, t^\prime] \rightarrow (0, \infty)$, there is some
$r^\prime = r^\prime(\epsilon, r) > 0$
with the following property.
Let $\M$ be a Ricci flow spacetime
such that $R \: : \: \M_{\le t^\prime} \rightarrow \R$ is
bounded below and proper, and
there is a compact set $K \subset \M_{\le t^\prime}$ so that for
each $m \in \M_{\le t^\prime} - K$,
\begin{enumerate}
\renewcommand{\theenumi}{\alph{enumi}}
\item The Hamilton-Ivey pinching condition of
  \cite[(A.14)]{kleiner_lott_singular_ricci_flows} is satisfied at
$m$, with time parameter $\t(m)$, and
\item The $r$-canonical neighborhood assumption of
  \cite[Appendix A.8]{kleiner_lott_singular_ricci_flows} is
satisfied at $m$.
\end{enumerate}
Then the conditions of Definition \ref{srf}
hold on $\M_{\le t^\prime}$, with parameters $\epsilon$, $\kappa^\prime$ and
$r^\prime$.
\end{proposition}
\begin{proof}
  Condition (a) of Definition \ref{srf}
holds on $\M_{\le t^\prime}$ by assumption.

Also by assumption, for $m \in \M_{\le t^\prime} - K$, 
the curvature operator at $m$
lies in the convex cone of \cite[(A.13)]{kleiner_lott_singular_ricci_flows}. The proof of Hamilton-Ivey
pinching in \cite[Pf. of Theorem 6.44]{chow_lu_ni_hamiltons_ricci_flow}, using
the vector-valued maximum principle, now goes through
since any violations
in $\M_{\le t^\prime}$ of
\cite[(A.14)]{kleiner_lott_singular_ricci_flows} would have to occur in $K$.
This shows that condition (b) of Definition \ref{srf}
holds on $\M_{\le t^\prime}$.

Since the $r$-canonical neighborhood assumption holds on
$\M_{\le t^\prime} - K$ the proof of \cite[Lemma 5.13]{kleiner_lott_singular_ricci_flows} 
shows that
$g_{\M}^{qp}$ is complete on $\M_{\le t^\prime}$ away from the time-zero slice
and the time-$t^\prime$ slice.
Proposition \ref{newkappa} now implies that 
$\M_{\le t^\prime}$
is $\kappa^\prime$-noncollapsed at scales less
than $\epsilon$.

To show that condition (c) of Definition \ref{srf}
holds on $\M_{\le t^\prime}$, with parameters $\epsilon$ and $\kappa^\prime$,
and some parameter $r^\prime > 0$,
we apply the method of proof of
\cite[Theorem 52.7]{kleiner_lott_perelman_notes} for smooth Ricci flow solutions.
Suppose that there is no such $r^\prime$.
Then there is a sequence $\{\M^k\}_{k=1}^\infty$ of Ricci flow spacetimes
satisfying the assumptions of the proposition, and a sequence
$r^\prime_k \rightarrow 0$, so that for each $k$
there is some $m_k \in \M^k_{\le t}$ where the 
$r^\prime_k$-canonical neighborhood assumption does not hold.
The first step in
\cite[Pf. of Theorem 52.7]{kleiner_lott_perelman_notes} 
is to find a point of violation so
that there are no nearby points of violation with much larger scalar
curvature, in an earlier time interval which is long in a scale-invariant
sense. The proof of this first step uses point selection. 
Because of
our assumption that the $r$-canonical neighborhood assumption
holds in $\M^k_{\le t^\prime} - K^k$, as soon as
$r^\prime_k < r(t^\prime)$ we know that any point of violation lies in $K^k$.
Thus this point selection argument goes through.
The second step in \cite[Pf. of Theorem 52.7]{kleiner_lott_perelman_notes}
is a bounded-curvature-at-bounded-distance statement that uses
Hamilton-Ivey pinching and $\kappa^\prime$-noncollapsing.  Since
we have already proven that the latter two properties hold, the
proof of the second step goes through.  The third and
fourth steps in \cite[Pf. of Theorem 52.7]{kleiner_lott_perelman_notes}
involve constructing an approximating $\kappa^\prime$-solution.  These last
two steps go through without change.
\end{proof}

Proposition \ref{asympprop} shows that the $r^\prime$-canonical neighborhood
assumption holds with parameter $r^\prime = r^\prime(t^\prime)$.
We can assume that $r^\prime$ is a decreasing function of $t^\prime$.
Hence $\M$ is a singular Ricci flow with parameters
$\epsilon$, $\kappa^\prime = \kappa^\prime(\epsilon)$ and
$r^\prime = r^\prime(\epsilon, r)$.

\section{Dimension of the set of singular times} \label{sec6}

In this section we give an upper bound on the Minkowski dimension of
the set of singular times for a Ricci flow spacetime.

The geometric input comes from the proofs of Propositions
\ref{prop_r_in_lp} and \ref{holder}.
We isolate it in the following lemma. The lemma states that any point with large curvature determines a region in backward spacetime on which the scalar curvature behaves nicely (i.e. $R^{-1}$ grows with upper and lower linear bounds as one goes backward in time), and which carries a controlled amount of volume.
\begin{lemma}
\label{lem_controlled_product_domain}
For every $\la >0$, $t<\infty$ there is a constant $C=C(\la,t)<\infty$ with the following property.

Let $\M$ be a singular Ricci flow and suppose $x\in \M_t$ is a point with $\rho(x):=R^{-\frac12}(x)\leq C^{-1}r(t)$.  Then  there is a product domain $U\subset \M$ defined on the time interval $[t_-,t]$, where $t_-:=t-C^{-1}r^2(t)$, with the following properties:
\ben
\item $U_t\subset B(x,C \rho(x))$.
\item (Scalar curvature control) For all $t'\in [t_-,t]$, $x'\in U_{t'}$, we have
\begin{equation}
\label{eqn_scalar_curvature_control}
C^{-1}R^{-1}(x)+\eta_-(t-t')\leq R^{-1}(x')\leq C R^{-1}(x)+\eta_+(t-t')\,.
\end{equation}
Here $\eta_\pm$ are constants coming from the geometry of $\kappa$-solutions.
\item (Volume control)
For $t'\in [t-C^{-1}r^2(t),t-\frac12 C^{-1}r^2(t)]$ we have $\vol(U_{t'})\geq C^{-1}r^{2-\la}(t)\rho^{1+\la}(x)$.  In particular the spacetime volume of $U$ is at least $\frac12 C^{-2}r^{4-\la}(t)\rho^{1+\la}(x)$.
\een
\end{lemma}
\begin{proof}
  The proof of the lemma is based on arguments similar to those in the proofs of Propositions \ref{prop_r_in_lp} and \ref{holder}.
  We give an outline of the proof.  The details are similar to those for
  Propositions \ref{prop_r_in_lp} and \ref{holder}.

{\em Case 1.  $x$ is sufficiently neck-like that we can apply the neck stability result.}  Then we let $U$ be a product domain with $U_t=B(x,\rho(x))$.  The scalar curvature estimate (\ref{eqn_scalar_curvature_control}) then follows from the fact that the worldline of every $y\in U_t$ remains necklike until its scale becomes comparable to the canonical neighborhood scale.  The volume estimates in (3) follow using the Jacobian estimate, as in \cite[Section 5]{kleiner_lott_singular_ricci_flows} or in the proof of Proposition~\ref{prop_r_in_lp}.

{\em Case 2. The canonical neighborhood of $x$ is neither sufficiently neck-like, nor nearly round.} Then for a constant $c$ independent of $x$,   we find that $P(x,c\rho(x))\cap \M_{t-c\rho^2(x)}$ contains a necklike point $y$ to which the previous argument applies.  If the  product domain $U_y$ associated with $y$ is defined on the time interval $[t_0,t-c\rho^2(x)]$, then we let $U$ be the result of extending $U_y$ to the interval $[t_0,t]$.

{\em Case 3.  The canonical neighborhood of $x$ is nearly round, i.e. $B:=B(x,100\rho(x))$ is nearly isometric to a spherical space form, modulo rescaling.}  We follow this region backward in time, and have two subcases:

3(a).  The region remains nearly round until its scale becomes comparable to $r(t)$.  Then we take $U$ to be the product region with $U_t=B$, and the scalar curvature and volume estimates follow readily from the fact that time slices of $U$ are nearly round.

3(b). For some $t_0<t$, and every $t'\in [t_0,t]$, the image $B_{t'}$ of $B$ in  $\M_{t'}$ under the flow of the time vector field $\D_t$ is $\delta$-close to round, but $B_{t_0}\subset\M_{t_0}$ is not $\frac{\delta}{2}$-close to round.   Then we can apply Case 2 to $B_{t_0}$ to obtain a product region $U'$, and we define $U$ by extending $U'$ forward in time over the time interval $[t_0,t]$.
\end{proof}

As a corollary of this lemma, we get:


\begin{theorem}
\label{thm_singular_time_dimension}
If $\M$ is a singular Ricci flow, $T<\infty$, then the set of times $t\in [0,T]$ such that $\M_t$ is noncompact has Minkowski dimension $\leq \frac12$.
\end{theorem}
\begin{proof}
Choose $\la>0$, and let $C=C(\la,T)<\infty$ be the constant from Lemma~\ref{lem_controlled_product_domain}.

Pick $A>C_\la^2r^{-2}(T)$.  Let $\T_A$ be the set of times $t \in[0,T]$ such that the time slice $\M_t$ contains a point with $R>A$.  

Let $\{t_i^0\}_{i\in I}$ be a maximal $A^{-1}$-separated subset of $\T_A$.  For every $i\in I$, we may find $t_i\in [0,T]$ with 
\begin{equation}
\label{eqn_ti_ti0_separation}
|t_i-t_i^0|\leq \const A^{-1}
\end{equation} such that $\M_{t_i}$ contains a point $x_i$ with $R(x_i)=A$;
the existence of such a point $t_i$ follows by iterating \cite[Lemma 3.3]{kleiner_lott_singular_ricci_flows}.  Now we apply Lemma~\ref{lem_controlled_product_domain} to $x_i$, for every $i\in I$, to obtain a collection $\{U_i\}_{i\in I}$ of product domains in $\M$.  

Note that if $i,j\in I$ and $U_i\cap U_j\neq\emptyset$, then comparing the scalar curvature using Lemma~\ref{lem_controlled_product_domain}(2), we get that $|t_i-t_j|<C_1A^{-1}$ for some $C_1=C_1(\la,T)$.  Hence the collection $\{U_i\}_{i\in I}$ has intersection multiplicity $<N=N(\la,T)$.  Now Lemma~\ref{lem_controlled_product_domain}(3) implies that the spacetime volume of each $U_i$ is at least
$$
\frac12 C^{-2}r^{4-\la}(T)A^{-\frac12(1+\la)}\,.
$$
Using the multiplicity bound and the bound on spacetime volume we get
$$
|I| \leq  C_2A^{\frac12 (1+\la)}\,,
$$
for $C_2=C_2(\la,T)$.  Since by (\ref{eqn_ti_ti0_separation}) we can cover $\T_A$ with at most $|I|$ intervals of length comparable to $A^{-1}$, this implies that $\cap_{A>0}\T_A$ has Minkowski dimension $\leq \frac12 +\frac{\la}{2}$.  As $\la$ is arbitrary, this proves the theorem.
\end{proof}

\bibliography{ricci_limits2}{}

\providecommand{\bysame}{\leavevmode\hbox to3em{\hrulefill}\thinspace}
\providecommand{\MR}{\relax\ifhmode\unskip\space\fi MR }
\providecommand{\MRhref}[2]{%
  \href{http://www.ams.org/mathscinet-getitem?mr=#1}{#2}
}
\providecommand{\href}[2]{#2}
\begin{thebibliography}{CLN06}

\bibitem[BKa]{bamler_kleiner_gsc}
R.~Bamler and B.~Kleiner, \emph{{Ricci flow and diffeomorphism groups of
  $3$-manifolds}}, http://lanl.arxiv.org/abs/1712.06197.

\bibitem[BKb]{bamler_kleiner_uniqueness}
\bysame, \emph{{Uniqueness and stability of Ricci flow through singularities}},
  http://lanl.arxiv.org/abs/1709.04122.

\bibitem[CLN06]{chow_lu_ni_hamiltons_ricci_flow}
B.~Chow, P.~Lu, and L.~Ni, \emph{Hamilton's {R}icci flow}, Graduate Studies in
  Mathematics, vol.~77, American Mathematical Society, Providence, RI; Science
  Press, New York, 2006.

\bibitem[KL08]{kleiner_lott_perelman_notes}
B.~Kleiner and J.~Lott, \emph{Notes on {P}erelman's papers}, Geom. Topol.
  \textbf{12} (2008), no.~5, 2587--2855.

\bibitem[KL17]{kleiner_lott_singular_ricci_flows}
\bysame, \emph{{Singular Ricci flows I}}, Acta Math. (2017), no.~219, 65--134.

\bibitem[Per02]{perelman_entropy}
G.~Perelman, \emph{The entropy formula for the ricci flow and its geometric
  applications}, http://arxiv.org/abs/math/0211159, 2002.

\bibitem[Per03]{perelman_surgery}
\bysame, \emph{Ricci flow with surgery on three-manifolds},
  http://arxiv.org/abs/math/0303109, 2003.

\end{thebibliography}
\bibliographystyle{amsalpha}

\end{document}